\documentclass{amsart}
\usepackage{amsmath,amsthm,amsfonts,amssymb,amscd,amsbsy}
\usepackage{caption,enumerate}
\usepackage{dsfont,lscape}
\usepackage[all]{xy}
\usepackage{hyperref}

\hypersetup{
    pdftoolbar=true,
    pdfmenubar=true,
    pdffitwindow=false,
    pdfstartview={FitH},
    pdftitle={},
    pdfauthor={},
    pdfsubject={},
    pdfkeywords={}
    pdfnewwindow=true,
    colorlinks=true, %set this false for printable version
    linkcolor=blue,
    citecolor=blue,
    urlcolor=black,
}

\newcommand{\Spec}{\mathrm{Spec}}

\newcommand{\id}{\operatorname{Id}}
\newcommand{\Vol}{\operatorname{Vol}}

\newcommand{\Ric}{\operatorname{Ric}}

\newcommand{\dist}{\operatorname{dist}}
\newcommand{\Area}{\operatorname{Area}}
\renewcommand{\Vol}{\operatorname{Vol}}
\newcommand{\Diff}{\operatorname{Diff}}
\newcommand{\Emb}{\operatorname{Emb}}
\newcommand{\imorse}{i_\text{\rm Morse}}

\newcommand{\csch}{\operatorname{csch}}
\newcommand{\sech}{\operatorname{sech}}

\newcommand{\Ss}{\mathds S}
\newcommand{\Hr}{\mathds H}
\newcommand{\Kr}{\mathds K}

\newcommand{\N}{\mathds N}
\newcommand{\Z}{\mathds Z}
\newcommand{\R}{\mathds R}

\newcommand{\C}{\mathds C}
\newcommand{\Ca}{\mathds C\mathrm a}

\newcommand{\GL}{\mathsf{GL}}

\newcommand{\SU}{\mathsf{SU}}

\newcommand{\U}{\mathsf{U}}
\newcommand{\Ut}{\mathsf{U}}
\newcommand{\Sp}{\mathsf{Sp}}
\newcommand{\Spin}{\mathsf{Spin}}
\newcommand{\T}{\mathsf{T}}
\newcommand{\G}{\mathsf{G}}
\newcommand{\K}{\mathsf{K}}
\renewcommand{\H}{\mathsf H}
\newcommand{\F}{\mathsf F}
\newcommand{\LL}{\mathsf{L}}

\newcommand{\so}{\mathfrak{so}}

\newcommand{\spin}{\mathfrak{spin}}

\DeclareMathOperator{\Ad}{Ad}

\DeclareMathOperator{\spec}{Spec}
\DeclareMathOperator{\Hom}{Hom}

\DeclareMathOperator{\tr}{tr}
\DeclareMathOperator{\diag}{diag}
\newcommand{\op}{\operatorname}
\newcommand{\Id}{\id}

\newcommand{\mi}{\mathrm{i}}

\DeclareMathOperator{\Cas}{Cas}
\newcommand{\inner}[2]{\langle {#1},{#2}\rangle }
\newcommand{\innerdots}{\langle {\cdot},{\cdot}\rangle }

\newcommand{\g}{\mathrm g}

\newcommand{\x}{\mathbf x}

\newcommand{\ee}{\varepsilon}
\newcommand{\re}{\op{Re}}
\newcommand{\fa}{\mathfrak a}

\newcommand{\ff}{\mathfrak f}
\newcommand{\fg}{\mathfrak g}
\newcommand{\fh}{\mathfrak h}

\newcommand{\fk}{\mathfrak k}
\newcommand{\fl}{\mathfrak l}
\newcommand{\fm}{\mathfrak m}

\newcommand{\fp}{\mathfrak p}
\newcommand{\fq}{\mathfrak q}

\newcommand{\ft}{\mathfrak t}

\allowdisplaybreaks

\newtheorem{theorem}{Theorem}[]
\newtheorem{lemma}[theorem]{Lemma}
\newtheorem{proposition}[theorem]{Proposition}
\newtheorem{corollary}[theorem]{Corollary}

\newtheorem{mainthm}{\sc Theorem}

\theoremstyle{definition}
\newtheorem{definition}[theorem]{Definition}
\newtheorem{notation}[theorem]{Notation}

\theoremstyle{remark}
\newtheorem{remark}[theorem]{Remark}

\title[Full spectrum of distance spheres in symmetric spaces of rank one]{Full Laplace spectrum of distance spheres in symmetric spaces of rank one}


\author[R.~G.~Bettiol]{Renato G.~Bettiol}
\address{City University of New York (Lehman College) \newline
\indent Department of Mathematics  \newline
\indent 250 Bedford Park~Blvd W\newline
\indent Bronx, NY, 10468, USA }
\email{r.bettiol@lehman.cuny.edu}

\author[E.~A.~Lauret]{Emilio~A.~Lauret}
\address{
	Universidad Nacional del Sur (UNS) - CONICET\newline
	\indent Instituto de Matem\'atica de Bah\'ia Blanca (INMABB) \newline
	\indent Departamento de Matem\'atica\newline
	\indent Av.\ Alem 1255, Bah\'ia Blanca B8000CPB, Argentina}
\email{emilio.lauret@uns.edu.ar}

\author[P.~Piccione]{Paolo Piccione}
\address{Universidade de S\~ao Paulo \newline
\indent Departamento de Matem\'atica \newline
\indent Rua do Mat\~ao, 1010 \newline
\indent S\~ao Paulo, SP, 05508-090, Brazil}
\email{piccione@ime.usp.br}

\allowdisplaybreaks
\numberwithin{equation}{section}
\numberwithin{theorem}{section}

\date{\today}

\thanks{The first-named author is supported by the National Science Foundation (DMS-1904342). The second-named author is supported by FonCyT (BID-PICT 2018-02073) and the Alexander von Humboldt Fountation (return fellowship). The third-named author is supported by Fapesp (2016/23746-6 and 2019/09045-3).} 

\begin{document}

\begin{abstract}
We use Lie-theoretic methods to explicitly compute the full spectrum of the Laplace--Beltrami operator on homogeneous spheres which occur as geodesic distance spheres in (compact or noncompact) symmetric spaces of rank one, and provide a single unified formula for all cases.
As an application, we find all resonant radii for distance spheres in the compact case, i.e., radii where there is bifurcation of embedded constant mean curvature spheres, and show that distance spheres are stable and locally rigid in the noncompact case.
\end{abstract}

\maketitle

\section{Introduction}\label{sec:introduction}

The family of (simply-connected) symmetric spaces of rank one consists of 
spheres and projective spaces $\Ss^n$, $\C P^n$, $\Hr P^n$, $\Ca P^2$, together with their noncompact duals, the hyperbolic spaces $H^n$, $\C H^n$, $\Hr H^n$, $\Ca H^2$. 
As Riemannian manifolds, these are \emph{two-point homogeneous spaces}, that is, any two pairs of points at the same distance can be mapped to one another by an isometry. 
In particular, their distance spheres
\begin{equation*}
S(r)=\big\{x\in M: \dist(x_0,x)=r \big\},
\end{equation*}
are homogeneous spaces themselves, and two distance spheres are isometric if and only if they have the same radius, regardless of their centers.
These homogeneous spheres are the main object of study in this paper;
in which we shall use Lie theory to explicitly compute their entire Laplace spectrum, and 
determine their stability (or lack thereof) as constant mean curvature hypersurfaces.

While distance spheres $S(r)$ in $\Ss^n$ and $H^n$ have constant curvature, i.e., are isometric to round spheres, just like in $\R^n$, this is no longer the case in projective and hyperbolic spaces. 
Geometrically, $S(r)\subset M$ are obtained by rescaling the unit round metric in the vertical direction(s) of the corresponding Hopf bundle by $t>0$:
\begin{equation}\label{eq:hopfbundles}
\begin{aligned}
\Ss^1_t\longrightarrow & \;\big(\Ss^{2n+1},\mathbf g(t)\big)\longrightarrow \C P^n, & \quad \text{ if } & M = \C P^{n+1} \text{ or } \,\C H^{n+1};\\
\Ss^3_t\longrightarrow & \;\big(\Ss^{4n+3},\mathbf h(t)\big)\longrightarrow\Hr P^n, & \quad \text{ if } & M = \Hr P^{n+1} \text{ or } \,\Hr H^{n+1}; \\
\Ss^7_t\longrightarrow & \;\big(\Ss^{15},\mathbf k(t)\big)\longrightarrow\Ss^8_{1/2}, & \quad \text{ if } & M = \Ca P^{2} \text{ or } \, \Ca H^2;
\end{aligned}
\end{equation}
where $\Ss^\ell_t$ denotes the $\ell$-dimensional sphere of constant curvature $\sec=1/t^2$, and then globally rescaling all directions by $\alpha>0$. With the convention (used throughout this paper) that the above projective and hyperbolic spaces with their canonical metrics have sectional curvatures $1\leq\sec_M\leq4$ and $-4\leq \sec_M\leq -1$ respectively, the values of $t$ and $\alpha$ for $S(r)\subset M$ are related to its geodesic radius $r$ as follows:
\begin{equation}\label{eq:t&alpha}
\begin{aligned}
&t=\cos r &\text{ and }\;\; &\alpha=\sin r, & \;  &0<r<\pi/2, &\text{if }M\text{ is a projective space}; \\
&t=\cosh r &\text{ and }\;\; &\alpha=\sinh r, & \; &r>0, &\text{if }M\text{ is a hyperbolic space}. \\
\end{aligned}
\end{equation}
Note that, with these conventions, the above projective spaces have diameter $\pi/2$.
Of course, all $S(r)$ become asymptotically round as $r\searrow0$, that is, they converge (up to homothety by $\alpha$) to the \emph{unit} round metric, which corresponds to $t=1$ in each of the families $\mathbf g(t)$, $\mathbf h(t)$, and $\mathbf k(t)$. Furthermore, only the metrics with either $t<1$ or $t>1$ appear (up to homotheties) as distance spheres $S(r)\subset M$, according to whether $M$ is projective or hyperbolic. 

It is convenient to refer to the Riemannian submersions \eqref{eq:hopfbundles} collectively as
\begin{equation}\label{eq:universal-hopf}
\Ss^{2d-1}_t\longrightarrow \big(\Ss^{N-1},\g(t)\big)\longrightarrow \mathds K P^n,
\end{equation}
where $\mathds K\in \{\C,\Hr,\Ca\}$, $d=\dim_\C \mathds K\in\{1,2,4\}$, $n\geq1$, and $N=2d(n+1)=\dim M$ is the (real) dimension of the ambient space $\mathds K P^{n+1}$ or $\mathds K H^{n+1}$. Recall that if $\mathds K=\Ca$, i.e., $d=4$, only $n=1$ is possible due to the non-associativity of Cayley numbers \cite{baez,lackmann}, and \eqref{eq:universal-hopf} is \emph{not} a homogeneous fibration \cite{gluck-ziller,guijarro}. 

Since the fibers of \eqref{eq:universal-hopf} are totally geodesic, the projection map ``commutes'' the Laplace--Beltrami operators of total space and base. In particular, lifting a Laplace eigenfunction of $\mathds K P^n$ produces a Laplace eigenfunction of $\big(\Ss^{N-1},\g(t)\big)$, with the same eigenvalue. 
Such eigenvalues are called \emph{basic}, and are independent of~$t$.
Although it has been known for a long time that all eigenvalues are sums of basic eigenvalues with certain Laplace eigenvalues of the fiber~\cite{Berard-BergeryBourguignon82,besson-bordoni}, determining exactly which sums of eigenvalues from $\Kr P^n$ and $\Ss^{2d-1}$ indeed appear in the spectrum of the total space can be somewhat impractical. 
We circumvent this with an alternative Lie-theoretic approach based on \cite{MutoUrakawa80}, recently used in \cite{Lauret-SpecSU(2),blp-firsteigenvalue} and expanded in Section~\ref{sec:homspectra} below, which yields our first main result:

\begin{mainthm}\label{mainthm:A}
The spectrum of the Laplace--Beltrami operator on the homogeneous sphere $\big(\Ss^{N-1},\g(t)\big)$, $N=2d(n+1)$, as in \eqref{eq:universal-hopf}, consists of the eigenvalues
\begin{equation}\label{eqthm:lambdapq-universal}
\lambda^{(p,q)}(t)=4p\big(p+q+d(n+1)-1\big)+2dnq+q(q+2d-2)\tfrac{1}{t^2}, \quad p,q\in\N_0,
\end{equation}
which are basic if $q=0$, and have multiplicity
\begin{equation}\label{eqthm:dpq-universal}
m_{p,q}=\frac{2p+q+d(n+1)-1}{d(n+1)-1}\frac{\binom{p+q+d(n+1)-2}{p+q}\binom{p+dn-1}{p}}{\binom{p+q+d-1}{p+q}}\,\chi(d,q),
\end{equation}
where $\chi(d,q)=\left(1+\frac{q}{d-1}\right)\frac{\Gamma(q+2d-2)}{\Gamma(q+1)\Gamma(2d-2)}$.
If different pairs $(p,q)$ give the same value $\lambda^{(p,q)}(t)$, the multiplicity of that eigenvalue is the sum of all the corresponding~$m_{p,q}$.
\end{mainthm}

We take the convention that $\chi(d,q)$ is extended by continuity to its removable singularity at $d=1$, i.e.,
\begin{equation*}
\chi(1,q)=\lim_{d\to1} \left(1+\frac{q}{d-1}\right)\frac{\Gamma(q+2d-2)}{\Gamma(q+1)\Gamma(2d-2)}=\begin{cases}
1 & \text{ if }q=0\\
2 & \text{ if }q\geq1,
\end{cases}
\end{equation*}
since $\Gamma(z)$ has a simple pole at $z=0$ of residue $1$, and $\Gamma(a)=(a-1)!$ for all $a\in\N$. Moreover, if $d\geq2$, note that $\chi(d,q)=\big(1+\frac{q}{d-1}\big)\binom{q+2d-3}{q}$ for all $q\in\N_0$. As usual, we agree that $\binom{a}{b}=0$ if $a<b$. 
Despite the convenient unified formulae \eqref{eqthm:lambdapq-universal} and \eqref{eqthm:dpq-universal}, the proof of Theorem~\ref{mainthm:A} is done analyzing each case $\Kr\in\{\C,\Hr,\Ca\}$ separately, and corresponding formulae can be found in Table~\ref{tab:eigenvalues}.

Note that setting $t=1$ in \eqref{eqthm:lambdapq-universal}, the eigenvalues $\lambda^{(p,q)}(1)=k(k+N-2)$, $k\in\N_0$, of the unit round sphere $\Ss^{N-1}$ are recovered, with $k=2p+q$. Moreover, \eqref{eqthm:dpq-universal} and combinatorial identities show that its multiplicity $\binom{k+N-1}{N-1}-\binom{k+N-3}{N-1}$ is equal to the sum of $m_{p,q}$ over all $p,q\in\N_0$ satisfying $2p+q=k$. Similarly, setting $q=0$, one recovers the eigenvalues $\lambda^{(p,0)}(t)=4p(p+d(n+1)-1)$, $p\in\N_0$, of the projective space $\Kr P^n$ and the corresponding multiplicities.

Several partial descriptions of the spectra in Theorem~\ref{mainthm:A} appear in the literature, e.g.~\cite{tanno1,tanno2,docarmo88,bp-calcvar}; in particular, the \emph{first} (nonzero) eigenvalue was computed in \cite{bp-calcvar}, see also \cite{blp-firsteigenvalue}. However, to the best of our knowledge, the \emph{full} Laplace spectrum cannot be directly extracted from these earlier results.

Using Theorem~\ref{mainthm:A} and \eqref{eq:t&alpha}, the full spectrum of the Laplace--Beltrami operator on any distance sphere $S(r)$ in a rank one symmetric space $M$ can be easily computed, since $\Delta_{\alpha \g}=\tfrac{1}{\alpha}\Delta_\g$. Although the lowest dimensional cases are excluded as $n\geq1$ in \eqref{eq:hopfbundles}, these are trivial since $\C P^1\cong \Ss^2(\tfrac12)$ and $\Hr P^1\cong \Ss^4(\tfrac12)$ are isometric to round spheres with $\sec_M=4$; and $\C H^1\cong H^2(\tfrac12)$ and $\Hr H^1\cong H^4(\tfrac12)$ are isometric to real hyperbolic spaces with $\sec_M=-4$, so distance spheres $S(r)\subset M$ in any of these spaces are just round spheres.

The spectrum of distance spheres is closely related to the local ambient geometry, e.g., it detects whether a harmonic space is locally symmetric~\cite{am-s}.
One of its global consequences is explored in our second main result, concerning the existence of other embedded constant mean curvature (CMC) spheres near distance spheres. More precisely, a distance sphere $S(r_*)\subset M$ is \emph{resonant} if there exists a sequence $r_j$ of radii converging to $r_*$ and a sequence $\Sigma_j\subset M$ of embedded spheres converging to $S(r_*)$, with constant mean curvature $H(\Sigma_j)=H(S(r_j))$, that are not congruent to $S(r_j)$. Note that $S(r_*)$ is \emph{non-resonant} if and only if, up to ambient isometries, $S(r)$ are locally the only embedded CMC spheres with their mean curvature if $r$ is sufficiently close to $r_*$. Recall that a hypersurface $\Sigma\subset M$ has constant mean curvature $H$ if and only if it is a stationary point for the functional $\Area(\Sigma)+H\,\Vol(\Sigma)$, where $\Area(\Sigma)$ is the $(N-1)$-volume of $\Sigma$ and $\Vol(\Sigma)$ is the $N$-volume of the region enclosed by $\Sigma$ in $M$, and $\Sigma$ is \emph{stable} if it is locally a minimum.

\begin{mainthm}\label{mainthm:B}
The distance spheres $S(r)$ in the projective spaces $\C P^{n+1}$, $\Hr P^{n+1}$, $n\geq1$, and $\Ca P^2$ are resonant if and only if $r=r_p$ for some $p\in\N$, where
\begin{equation*}
r_p:=\arctan\sqrt{\frac{4p(p-1) + N(2p-1)+1}{2d-1}},
\end{equation*}
$d=\dim_\C \Kr\in\!\{1,2,4\}$ per $\Kr\in\{\C,\Hr,\Ca\}$, and $N=\dim \Kr P^{n+1}=2d(n+1)$. On the other hand, for all $r>0$, the distance spheres $S(r)$ in the hyperbolic spaces $\C H^{n+1}$, $\Hr H^{n+1}$, $n\geq1$, and $\Ca H^2$ are stable and non-resonant.
\end{mainthm}

The existence of infinitely many resonant distance spheres in $\C P^{n+1}$ and $\Hr P^{n+1}$ with radii accumulating at $\pi/2$ had been established in \cite{bp-imrn}. Nevertheless, the coarser equivariant spectral methods used there do not allow one to explicitly determine which radii $0<r<\pi/2$ are resonant, nor to handle the case of $\Ca P^2$, since \eqref{eq:universal-hopf} is \emph{not} a homogeneous fibration if $\Kr=\Ca$.
 Moreover, it was known that $S(r)\subset \Kr P^{n+1}$ is stable if and only if $0<r<r_1=\arctan\sqrt\frac{N+1}{2d-1}$, see~\cite[Thms.~1.3, 1.4]{docarmo88}, and that $S(r)\subset \Kr H^{n+1}$ are stable for all $r>0$, see~\cite[Thm.~2]{riv-tom}.

The path leading from Theorem \ref{mainthm:A} to Theorem \ref{mainthm:B} is that the stability operator (or \emph{Jacobi operator}) for a CMC hypersurface $\Sigma\subset M$ is $J_\Sigma=\Delta_{\Sigma} - (\Ric(\vec n_\Sigma)+\|A_{\Sigma}\|^2)$, hence its spectrum is a shift of the Laplace spectrum of $\Sigma$ by a curvature term, which is constant if $\Sigma$ is a distance sphere $S(r)$ as above. Stability of $S(r)$ is equivalent to nonnegativity of the first eigenvalue of $J_{S(r)}$, while resonance of $S(r_*)$ is detected by eigenvalues of $J_{S(r)}$ crossing zero at $r=r_*$, see Section~\ref{sec:bif} for details.

This paper is organized as follows. Section~\ref{sec:homspectra} describes a Lie-theoretic approach tailored to compute the Laplace spectrum on the total space of Riemannian submersions such as \eqref{eq:hopfbundles}. The outcome for the first families in \eqref{eq:hopfbundles} is given in Sections~\ref{sec:fullspectra-odd-spheres} and \ref{sec:fullspectra-S^4n+3} respectively, while Section~\ref{sec:fullspectra-S^15} deals with the third case. These results are unified in Section~\ref{sec:unified}, with the proof of Theorem~\ref{mainthm:A}. The applications regarding resonance and rigidity are discussed in Section~\ref{sec:bif}, where Theorem~\ref{mainthm:B} is proven.

\section{Computing the Laplace spectrum of a homogeneous space}\label{sec:homspectra}

\subsection{Basic setting}\label{subsec:basicsetting}
Let $\H\subset \K\subset \G$ be compact Lie groups, with Lie algebras $\fh\subset\fk\subset\fg$. Fix a bi-invariant metric on $\G$, i.e., an $\Ad(\G)$-invariant inner product $\innerdots_0$ on $\fg$. For instance, a natural choice on most matrix Lie groups is
\begin{equation}\label{eq:inn0}
\inner{X}{Y}_0= -\tfrac12\re(\tr(XY)).
\end{equation}
Let $\fp$ and $\fq$ be the $\innerdots_0$-orthogonal complements of $\fh$ in $\fk$, and $\fk$ in $\fg$, so that $\fk=\fh\oplus \fp$ and $\fg=\fk\oplus \fq = \fh\oplus (\fp\oplus \fq)$ are Cartan decompositions. In particular, the $\H$-action on $\fp\oplus\fq$ via the adjoint representation of $\G$ is identified with the isotropy representation of $\G/\H$. 
Note that although $\fp$ and $\fq$ are subrepresentations, they need not be irreducible.
Consider the family of $\Ad(\H)$-invariant inner products 
\begin{equation}\label{eq:innprodrs}
\innerdots_{(r,s)} = \frac{1}{r^2}\, \innerdots_0\big|_{\fp} + \frac{1}{s^2} \, \innerdots_0\big|_{\fq}, \quad r,s>0,
\end{equation}
on $\fp\oplus\fq$, which induces a corresponding family of $\G$-invariant metrics $\g_{(r,s)}$ on~$\G/\H$. 

Up to homotheties, this is the \emph{canonical variation} of the Riemannian submersion
\begin{equation}\label{eq2:Riemsubm}
\K/\H \longrightarrow \G/\H \longrightarrow \G/\K
\end{equation}
where all spaces are endowed with normal homogeneous metrics induced by $\innerdots_0$. In geometric terms, $\g_{(r,s)}$ is obtained by rescaling the vertical and horizontal directions of \eqref{eq2:Riemsubm} by $1/r$ and $1/s$, respectively. If $\fp$ and $\fq$ are irreducible and non-equivalent as $\H$-modules, then any $\G$-invariant metric on $\G/\H$ is isometric to some~$\g_{(r,s)}$.

\subsection{The Lie-theoretic method}
In this section, we describe the Lie-theoretic procedure to compute the Laplace--Beltrami spectrum of $\left(\G/\H,\g_{(r,s)}\right)$, which relies on knowledge of representation branching rules involving $\G$, $\K$, and $\H$. The discussion below is based on \cite{MutoUrakawa80} and our earlier work \cite[\S2]{blp-firsteigenvalue}, and provides a computationally efficient alternative to more classical methods in \cite{Berard-BergeryBourguignon82,besson-bordoni}.

\begin{notation}
Given a compact Lie group $\mathsf J$, let $\widehat{\mathsf J}$ be its unitary dual, i.e., the set of equivalence classes of irreducible unitary representations of $\mathsf J$. We shall consider elements of $\widehat{\mathsf J}$ as representations $(\pi,V_\pi)$, i.e., homomorphisms $\pi\colon \mathsf J\to\GL(V_\pi)$.
Given $\mathsf J$-representations $(\sigma,V_\sigma)$ and $(\tau,V_\tau)$, set $[\sigma:\tau]:=\dim \Hom_{\mathsf J}(V_\sigma,V_\tau)$. 
Note that if $\sigma$ is irreducible, then $[\sigma:\tau]$ is the multiplicity of $\sigma$ in the decomposition of $\tau$ in irreducible components. In particular, $[\sigma:\tau]>0$ for only finitely many $\sigma\in \widehat {\mathsf{J}}$. 
\end{notation}

\begin{definition}
The set of \emph{spherical representations} associated to $(\G,\H)$ is $$\widehat \G_{\H} := \big\{(\pi,V_\pi) \in \widehat \G: V_\pi^\H\neq0\big\}=\big\{\pi\in\widehat \G: [1_\H:\pi|_{\H}]>0\big\},$$ where $V_\pi^\H$ is the subspace of $V_\pi$ given by $\H$-invariant elements, and $1_\H$ is the trivial representation of $\H$.
The \emph{Casimir element of $\fg$ with respect to $\innerdots_0$} is the element $\Cas_{\fg,\innerdots_0}:= X_1^2+\dots+X_{\dim\fg}^2$ of the universal enveloping algebra $\mathcal U(\mathfrak g_\C)$, where $\{X_1,\dots,X_{\dim\fg}\}$ is any $\innerdots_0$-orthonormal basis of $\fg$. 
Since we fixed an $\Ad(\G)$-invariant inner product $\innerdots_0$, we denote $\Cas_{\fg,\innerdots_0}$ by $\Cas_{\fg}$, and similarly for the Casimir elements $\Cas_{\fk}$ and $\Cas_{\fh}$ of $(\fk,\innerdots_0|_{\fk})$ and $(\fh,\innerdots_0|_{\fh})$.
\end{definition}

\begin{notation}
Let $(\varphi,V_\varphi)$ be a unitary representation of a compact Lie group $\mathsf{J}$. 
We shall also denote by $\varphi$ the induced representations of the Lie algebra $\mathfrak j$ of $\mathsf{J}$, of its complexification $\mathfrak j_\C:=\mathfrak j\otimes_\R\C$, and of its universal enveloping algebra $\mathcal U(\mathfrak j_\C)$.
Since $\varphi(a)\colon V_\varphi\to V_\varphi$ is unitary for all $a\in\mathsf{J}$, it follows that $\varphi(X)$ is skew-symmetric for all $X\in\fg$, and, consequently, $\varphi(-X^2)$ is self-adjoint for all $X\in\fg$.
\end{notation}

For $\pi\in\widehat \G$, the operator $\pi(\Cas_{\fg})\colon V_\pi\to V_\pi$ commutes with $\pi(a)$ for all $a\in\G$. 
Thus, by Schur's Lemma, $\pi(-\Cas_{\fg})= \lambda^\pi\, \Id_{V_\pi}$ for some $\lambda^\pi>0$. Analogously, $\tau(-\Cas_{\fk})= \lambda^\tau\, \Id_{V_\tau}$ and $\sigma(-\Cas_{\fh})= \lambda^\sigma\, \Id_{V_\sigma}$ for $\tau\in\widehat \K$ and $\sigma\in\widehat \H$. The constants $\lambda^\pi$ and $\lambda^\tau$ can be computed explicitly using Lie-theoretic objects, see Subsection~\ref{subsec:freud}.

\begin{theorem}\label{thm2:spec}
The spectrum of the Laplace--Beltrami operator on the homogeneous space $\left(\G/\H,\g_{(r,s)}\right)$ consists of the eigenvalues
\begin{equation}\label{eq2:lambda-pi-tau}
\lambda^{\pi,\tau}(r,s) = (r^2-s^2)\, \lambda^\tau  + s^2 \lambda^\pi,
\end{equation}
where $(\pi,\tau)\in \widehat \G_\H \times \widehat \K$ is such that $[\tau:\pi|_{\K}]>0$, with multiplicity 
\begin{equation*}
m_{\pi,\tau}= [1_\H:\tau|_{\H}] [\tau:\pi|_{\K}]\dim V_\pi.
\end{equation*}
Moreover, \eqref{eq2:lambda-pi-tau} is basic for the Riemannian submersion $\left(\G/\H,\g_{(r,s)}\right)\to \G/\K$ if $\pi\in \widehat \G_\K$ and $\tau=1_\K$, in which case $\lambda^{\pi,\tau}(r,s)=s^2 \lambda^\pi$ and $m_{\pi,\tau}=[1_\K:\pi|_{\K}]\dim V_\pi$.
\end{theorem}

\begin{proof}
Let $\{X_1,\dots,X_{{\dim \fp}}\}$ and $\{Y_1,\dots,Y_{{\dim \fq}}\}$ be $\innerdots_0$-orthonormal bases of $\fp$ and $\fq$ respectively, and note that $\{rX_1,\dots,rX_{{\dim \fp}}, sY_1,\dots,sY_{{\dim \fq}}\}$
is an orthonormal basis of $\fp\oplus\fq$ with respect to $\innerdots_{(r,s)}$. 
Set $C_{\fp}=X_1^2+\dots+X_{{\dim \fp}}^2$, $C_{\fq}=Y_1^2+\dots+ Y_{{\dim \fq}}^2$, and $C_{(r,s)}= t^2\,C_{\fp}+s^2\,C_{\fq}$.
According to \cite[Prop.~2.2]{blp-firsteigenvalue}, the spectrum of the Laplace--Beltrami operator on $\left(\G/\H, \g_{(r,s)}\right)$ is the union of eigenvalues of $\pi(-C_{(r,s)})|_{V_{\pi}^\H}$, where $\pi \in\widehat \G_\H$, each with multiplicity $\dim V_\pi$.

We need to show \eqref{eq2:lambda-pi-tau} appears in the spectrum of $\pi(-C_{(r,s)})|_{V_{\pi}^\H}$ with multiplicity $[1_\H:\tau|_{\H}] [\tau:\pi|_{\K}]$ for any $\tau\in\widehat \K$ satisfying $[\tau:\pi|_{\K}]>0$, and that these eigenvalues exhaust the spectrum. For $v\in V_\pi^\H$, we have that 
\begin{equation}\label{eq:pi(-C_rs)v}
\begin{aligned}
\pi(-C_{(r,s)})\cdot v 
&= r^2\,\pi(-C_{\fp})\cdot v + s^2 \,\pi(-C_{\fq})\cdot v
\\
&= (r^2-s^2)\,\pi(-C_{\fp})\cdot v + s^2\, \pi(-C_{\fp}-C_{\fq})\cdot v
\\
&= (r^2-s^2)\,\pi(-\Cas_{\fh}-C_{\fp})\cdot v + s^2 \,\pi(-\Cas_{\fh}-C_{\fp}-C_{\fq})\cdot v
\\
&= (r^2-s^2)\,\pi(-\Cas_{\fk})\cdot v + s^2\, \pi(-\Cas_{\fg})\cdot v.
\end{aligned}
\end{equation}
The third equality follows from $\pi(\Cas_\fh)\cdot v=0$, since $v$ is $\H$-invariant. 
While clearly $\pi(-\Cas_{\fg})\cdot v = \lambda^\pi v$, the computation of the term $\pi(-\Cas_{\fk})\cdot v$ is more involved.

Consider the decomposition 
\begin{equation*}
V_\pi= \bigoplus_{\tau\in\widehat \K, \, [\tau:\pi|_{\K}]>0} V_\pi(\tau),
\end{equation*}
where the subspace $V_\pi(\tau)$ is given by the sum of all $\K$-submodules of $V_\pi$ equivalent to $\tau$. 
As a $\K$-module, $V_\pi(\tau)$ is equivalent to $[\tau:\pi|_{\K}]$ copies of $\tau$. 
Since $V_\pi(\tau)$ is obviously invariant under the action of $\H$, we conclude that
\begin{equation*}
V_\pi^\H= \bigoplus_{\tau\in\widehat \K, \, [\tau:\pi|_{\K}]>0} V_\pi(\tau)^\H.
\end{equation*}
For $v\in V_\pi(\tau)^\H$, it follows that $\pi(-\Cas_{\fk})\cdot v = \lambda^{\tau}v$, and, consequently, from \eqref{eq:pi(-C_rs)v},
\begin{equation*}
\pi(-C_{(r,s)})\cdot v 
= \big((r^2-s^2)\, \lambda^\tau  + s^2 \lambda^\pi\big)\,  v = \lambda^{\pi,\tau}(r,s)\, v.
\end{equation*} 
Moreover, these eigenvalues exhaust the spectrum of $\pi(-C_{(r,s)})|_{V_\pi^\H}$, since
\begin{equation*}
\dim V_\pi^\H = \sum_{\tau\in\widehat \K,\,[\tau:\pi|_{\pi}]>0} \dim V_\pi(\tau)^\H
= [\tau:\pi|_{\K}]\dim V_\tau^{\H}= [1_\H:\tau|_{\H}][\tau:\pi|_{\K}].
\end{equation*}  

Finally, by definition, \eqref{eq2:lambda-pi-tau} is basic for the submersion $\left(\G/\H,\g_{(r,s)}\right)\to \G/\K$ if the associated eigenfunctions are constant along the fibers $\K/\H$. In this case, they descend to eigenfunctions of the Laplace--Beltrami operator on the base $\G/\K$. Applying \cite[Prop.~2.2]{blp-firsteigenvalue} to $\G/\K$, this corresponds to $\pi\in\widehat \G_\K$ and $\tau=1_\K$.
\end{proof}

Following the method described in Theorem~\ref{thm2:spec}, the ingredients needed to explicitly determine the Laplace spectrum of the homogeneous space $(\G/\H,\g_{(r,s)})$ are:
\begin{enumerate}[(i)]
\item the set $\widehat \G_\H$ of spherical representations associated to $(\G,\H)$;

\item the integers $[1_\H:\tau|_\H]$ and $[\tau:\pi|_{\K}]$ for $\tau\in\widehat \K$ satisfying $[\tau:\pi|_{\K}]>0$;

\item the coefficients $\lambda^\pi$ and $\lambda^\tau$ for all $\pi\in\widehat \G_{\H},\tau\in\widehat \K$ with $[1_\H:\tau|_\H][\tau:\pi|_{\K}]>0$. 
\end{enumerate}
All the above are Lie-theoretic in nature. While the first is known in many cases, the second depends on branching rules that are typically rather intricate, making this the most difficult part of the computation, see Subsection~\ref{subsec:prod}. Fortunately, the scalars $\lambda^\pi$ and $\lambda^\tau$ are easily computed using Freudenthal's formula, as follows.

\subsection{Freudenthal's formula}\label{subsec:freud}
Fix a maximal torus $\T$ in $\G$ such that $\T\cap\K$ and $\T\cap\H$ are maximal tori in $\K$ and $\H$, respectively. 
Then $\ft_\C=\ft\otimes_\R\C$ is a Cartan subalgebra of $\fg_\C=\fg\otimes_\R\C$, and we denote by $\Phi(\fg_\C,\ft_\C)$ its root system. If $\fg$ is not semisimple (but necessarily reductive), then $\Phi(\fg_\C,\ft_\C)$ is the root system associated to the semisimple part $[\fg,\fg]$ with respect to $[\fg,\fg]\cap\ft$. 
Fix an order on $\mi\ft^*$ inducing a positive root system $\Phi^+(\fg_\C,\ft_\C)$. 
By the Highest Weight Theorem (see e.g.~\cite[Thm 9.4, 9.5]{hall-book} or \cite[Thm~5.110]{Knapp-book-beyond}), irreducible $\G$-representations correspond to elements in the set $P^+(\G)$ of dominant $\G$-integral weights. 
Analogous objects are defined for $\K$ and $\H$, provided the orders are compatible. 
For $\Lambda\in P^{+}(\G)$, we denote by $\pi_\Lambda$ the unique (up to equivalences) irreducible representation of $\G$ with highest weight $\Lambda$. Analogously, for $\mu\in P^+(\K)$ and $\nu\in P^+(\H)$, 
we denote by $\tau_{\mu}$ and $\sigma_{\nu}$ the representations with highest weight $\mu$ and $\nu$, respectively.

Freudenthal's formula (see \cite[Lem.~5.6.4]{wallach-book} or \cite[Prop.~10.6]{hall-book}) applied to $\Lambda\in P^+(\G)$ and $\mu\in P^+(\K)$ gives, respectively,
\begin{equation}\label{eq:Casimirscalar}
\lambda^{\pi_\Lambda}=\langle \Lambda, \Lambda+ 2\rho_{\mathfrak g}\rangle_{0}, \quad \text{and}\quad 
\lambda^{\tau_\mu} = \langle \mu,\mu+ 2\rho_{\mathfrak k}\rangle_{0},
\end{equation}
where $\rho_{\mathfrak g}$ and $\rho_{\mathfrak k}$ are half the sum of positive roots in $\Phi^+(\mathfrak g_\C,\mathfrak t_\C)$ and $\Phi^+(\mathfrak k_\C, (\mathfrak t\cap\mathfrak k)_\C)$, respectively, and $\langle\cdot,\cdot\rangle_{0}$ is the Hermitian extension of $\innerdots_0|_{\mathfrak t}$ to $\mathfrak t_\C^*$.

\subsection{Product group}\label{subsec:prod}
The branching problem needed to compute ingredient (ii) above has an important simplification if $\K=\H\LL\simeq\H\times \LL$, where $\LL$ is a closed subgroup of $\G$ that commutes with $\H$.
In this case, the submersion \eqref{eq2:Riemsubm} becomes
\begin{equation*}
\LL \longrightarrow \G/\H \longrightarrow \G/(\H\times\LL).
\end{equation*}
It is well known that every irreducible $\K$-representation is of the form $\sigma\otimes \phi$ for some $\sigma\in\widehat \H$ and $\phi\in\widehat \LL$. 
Since $(\sigma\otimes\phi)|_{\H}=\sigma$, any $\tau\in\widehat \K$ contributing to $\spec\!\left(\G/\H, \g_{(r,s)}\right)$ in Theorem~\ref{thm2:spec} must be of the form $\tau=1_\H\otimes\phi$, and, also, $[1_\H:\tau|_\H]=1$. 
Moreover, 
\begin{equation*}
 [\tau:\pi|_{\K}] = [1_\H\otimes\phi :\pi|_{\K}] = \dim_{\H\times\LL}(V_{1_\H}\otimes V_{\phi}, V_\pi) = \dim_{\LL}(V_{\phi}, V_\pi^\H)=:[\phi:V_\pi^\H].
\end{equation*}
In other words, since $\H$ and $\LL$ commute, the $\LL$-action leaves $V_\pi^\H$ invariant, and $[1_\H\otimes\phi :\pi|_{\K}]$ is the multiplicity of $\phi$ in the decomposition of $V_\pi^\H$ as an $\LL$-module. 
Furthermore, by Freudenthal's formula, for any $\eta\in P^+(\LL)$, 
\begin{equation}\label{eq2:lambda-1_Hxphi}
\lambda^{1_\H\otimes \phi_{\eta}} = \langle \mu_{1_\H\otimes \phi_{\eta}}, \mu_{1_\H \otimes \phi_{\eta}}+2\rho_{\mathfrak h\oplus\mathfrak l} \rangle_0 = \langle \eta, \eta +2\rho_{\mathfrak l} \rangle_0,
\end{equation}
where 
$\fl$ is the Lie algebra of $\LL$, and, as before, $\rho_{\mathfrak l}=\frac12 \sum_{\alpha\in \Phi^+(\mathfrak l_\C, (\mathfrak t\cap\mathfrak l)_\C)} \alpha$.
Therefore, we may restate Theorem~\ref{thm2:spec} in this case as follows.

\begin{corollary}\label{cor:specHL}
	If $\K=\H\LL$ as above, then the spectrum of the Laplace--Beltrami operator on the homogeneous space $\left(\G/\H,\g_{(r,s)}\right)$ consists of the eigenvalues
	\begin{equation}\label{eq2:lambda-pi-tau-L}
	\lambda^{\pi,1_\H\otimes \phi_\eta}(r,s) = (r^2-s^2)\, \langle \eta, \eta +2\rho_{\mathfrak l} \rangle_0  + s^2 \lambda^\pi,
	\end{equation}
	where $(\pi,\phi_\eta)\in \widehat \G_\H \times \widehat \LL$ is such that $[\phi_\eta: V_\pi^{\H}]>0$, with multiplicity 
	\begin{equation*}
		m_{\pi,\phi_\eta}= [\phi_\eta: V_\pi^{\H}] \dim V_\pi.
	\end{equation*}
Moreover, \eqref{eq2:lambda-pi-tau-L} is basic for the Riemannian submersion $\left(\G/\H,\g_{(r,s)}\right)\to \G/\K$ if $\eta=0$, in which case $\lambda^{\pi,1_\H\otimes \phi_\eta}(r,s)=s^2 \lambda^\pi$ and $m_{\pi,\phi_\eta}=\dim V_\pi$.
\end{corollary}

\section{\texorpdfstring{Eigenvalues of the Laplacian on $\Ss^{2n+1}$}{Eigenvalues of the Laplacian}}\label{sec:fullspectra-odd-spheres}

In this section, we determine the full Laplace spectrum of the homogeneous spheres $\left(\Ss^{2n+1},\mathbf g(t)\right)$, $n\geq1$, as in \eqref{eq:hopfbundles}.
Although there are several partial results in the literature, e.g.,
this is done for all odd $n$ in~\cite[Thm~3.9, Rem~3.10]{blp-firsteigenvalue}, we include a complete argument below to illustrate the method in Section~\ref{sec:homspectra}.

A homogeneous metric on $\Ss^{2n+1}$ is $\SU(n+1)$-invariant if and only if it is $\Ut(n+1)$-invariant. 
Although the $\Ut(n+1)$-action on $\Ss^{2n+1}$ is not effective, we shall use it since it simplifies some computations. Throughout this section, we set:
\begin{equation}\label{eqSU:GKH}
\begin{aligned}
\G&= \Ut(n+1), &\LL&= \left\{ \begin{pmatrix} I & 0 \\ 0 & z\end{pmatrix}: z\in\Ut(1)\right\},\\
\H&= \left\{ \begin{pmatrix} A & 0 \\ 0 & 1\end{pmatrix}: A\in\Ut(n)\right\}, &
\K&=\left\{ \begin{pmatrix} A & 0 \\ 0 & z\end{pmatrix}: A\in\Ut(n),\, z\in\Ut(1)\right\}.
\end{aligned}
\end{equation}
Clearly, $\H\simeq \Ut(n)$, $\LL\simeq \Ut(1)$, and $\K=\H\LL\simeq \Ut(n)\Ut(1)$, as in Subsection~\ref{subsec:prod}, and it is well known that $\G/\H\cong \Ss^{2n+1}$ and $\G/\K\cong \C P^n$. 
It is easy to check that 
\begin{equation}\label{eq3:p-q}
\begin{aligned}
\fp& =\left\{ \begin{pmatrix} 0 & 0\\ 0 & \mi\theta \end{pmatrix}: \theta\in\R \right\}, & \quad
\fq& =\left\{ \begin{pmatrix} 0 & v \\ -  v^* & 0\end{pmatrix}: v\in \C^n \right\}
\end{aligned}
\end{equation}
satisfy $\fk=\fh\oplus\fp$ and $\fg=\fk\oplus \fq= \fh\oplus (\fp\oplus\fq)$. 
Moreover, as subrepresentations of the isotropy representation of $\H$, $\fp$ is trivial and $\fq$ is the standard representation.

Consider the $\G$-invariant metrics $\g_{(r,s)}$ on $\G/\H$ as in Subsection~\ref{subsec:basicsetting}. Since $\fp$ and $\fq$ are irreducible and non-equivalent, every $\G$-invariant metric on $\G/\H$ is isometric to some $\g_{(r,s)}$; e.g., for all $t>0$, the metric $\mathbf g(t)$ in \eqref{eq:hopfbundles} is isometric to $\g_{(\frac{1}{t\sqrt{2}},1)}$.

\begin{proposition}\label{thmSU:spec}
For all $n\geq1$, the spectrum of the Laplace--Beltrami operator on $\big(\Ss^{2n+1},\g_{(r,s)}\big)$ consists of the eigenvalues
\begin{equation}
\label{eqSU:lambda}
\lambda^{(p,q)}(r,s) =\big(4p(p+q+n)+2nq\big) s^2 +2q^2 r^{2}, \quad p,q\in\N_0,
\end{equation}
which are basic if $q=0$, and have multiplicity
\begin{equation}
\label{eqSU:d_pq}
m_{p,q} = {(2-\delta_{q0})} \frac{2p+q+n}{n} \binom{p+q+n-1}{p+q} \binom{p+n-1}{p}.
\end{equation}
% and, as usual, $\delta_{q0}=1$ if $q=0$ and $\delta_{q0}=0$ otherwise.
\end{proposition}

\begin{proof}
Since the groups \eqref{eqSU:GKH} satisfy $\K=\H\LL$, we may apply Corollary~\ref{cor:specHL}.
Fix the maximal torus of $\G$ given by 
$\T = \{\diag(e^{\mi \theta_1}, \dots,  e^{\mi \theta_{n+1}}) : \theta_1,\dots,\theta_{n+1}\in\R\}.$
Note that $\T\cap \K$, $\T\cap \H$, and $\T\cap \LL$ are maximal tori in $\K$, $\H$, and $\LL$ respectively. 
The Lie algebra $\mathfrak t$ and its complexification $\mathfrak t_\C$ consist of elements
$Y=\diag({\mi\theta_1},\dots, {\mi\theta_{n+1}}),$
where $\theta_j$ are in $\R$ and $\C$, respectively.
Define $\varepsilon_j\colon \mathfrak t_\C^*\to \C$ as $\varepsilon_j(Y)=\mi \theta_j$, for $Y$ as above, and note that $\big\{\frac{1}{\sqrt{2}} \varepsilon_1,\dots, \frac{1}{\sqrt{2}} \varepsilon_{n+1}\big\}$ is a $\innerdots_0$-orthonormal basis of $\mathfrak t_\C^*$. 

With the standard order, we have $\Phi^+(\mathfrak g_\C,\mathfrak t_\C) = \{\varepsilon_i-\varepsilon_j:1\leq i <j\leq n+1\}$, so half the sum of positive roots is $\rho_{\mathfrak g}= \sum_{j=1}^{n+1} \frac{n+2-j}{2}\varepsilon_j$, and the set of dominant integral weights is
$P^{+}(\G)=
\left\{\sum_{j=1}^{n+1} a_j \varepsilon_j \in\bigoplus_{j=1}^{n+1}\Z\varepsilon_j:a_1\geq a_2\geq \dots\geq a_{n+1} \right\}.$

The classical branching rule from $\G$ to $\H$ (see e.g.\ \cite[Thm.~9.14]{Knapp-book-beyond}) states that, if $\Lambda=\sum_{i=1}^{n+1} a_i\varepsilon_i\in P^+(\G)$ and $\nu=\sum_{i=1}^{n} b_i\varepsilon_i\in P^+(\H)$, then $[\sigma_\nu:\pi|_{\Lambda}]>0$ if and only if 
$a_1\geq b_1\geq a_2\geq \dots \geq a_{n}\geq b_n\geq a_{n+1};$
in which case $[\sigma_\nu:\pi|_{\Lambda}]=1$.
We conclude that $\pi_{\Lambda}\in\widehat\G_{\H}$, i.e., $\dim V_\pi^\H =[1_\H:\pi|_{\Lambda}]>0$, if and only if $a_i=0$ for all $2\leq i\leq n$ and $a_1\geq0\geq a_{n+1}$.
Therefore, the set of spherical representations is:
\begin{equation*}
\widehat \G_\H = \{ \pi_{k,l}:= \pi_{l\varepsilon_1-k\varepsilon_{n+1}}: k,l\in \N_0 \}.
\end{equation*}
We henceforth abbreviate $V_{k,l}:=V_{\pi_{k,l}}$. 
It is a simple matter to check that
\begin{equation}\label{eq:SUdimV_prelim}
\dim V_{k,l}=\frac{k+l+n}{n} \binom{k+n-1}{k} \binom{l+n-1}{l},
\end{equation}
by the Weyl Dimension Formula, see e.g.~\cite[Thm.~5.84]{Knapp-book-beyond}. 

Note that $\LL\simeq \Ut(1)$ is abelian and $(\fl\cap\ft)_\C^*= \C\varepsilon_{n+1}$.
Thus, its root system is empty, i.e.\ $\rho_{\fl}=0$, and every $\LL$-integral weight is dominant, that is, $P^+(\LL)=\{\phi_m:= \phi_{m\varepsilon_{n+1}} : m\in\Z\}$. 
It is well known that $V_{k,l}^\H\simeq \phi_{l-k}$ as $\LL$-modules; more precisely, $\pi_{k,l}\left( \left(\begin{smallmatrix} I_n & 0 \\ 0 & z\end{smallmatrix} \right)\right) \cdot v = z^{l-k}\, v$ for $v\in V_{{k,l}}^\H$ and $z\in\Ut(1)$, see e.g.~\cite[Thm.~8.1.2]{GoodmanWallach-book-Springer}.

From Corollary~\ref{cor:specHL}, the eigenvalues of $\left(\Ss^{2n+1},\g_{(r,s)}\right)$ are $\lambda^{\pi_{k,l},1_\H\otimes \phi_{l-k}}(r,s)$ for all $k,l\in\N_0$, with multiplicity $\dim V_{k,l}$. 
Moreover, by \eqref{eq:Casimirscalar} and \eqref{eq2:lambda-1_Hxphi}, we have that 
\begin{align*}
\lambda^{1_\H\otimes \phi_{m}} &= \langle m\varepsilon_{n+1}, m\varepsilon_{n+1} \rangle_0 = 2m^2, \\
\lambda^{\pi_{k,l}} &= \langle l\varepsilon_1- k\varepsilon_{n+1} +2\rho_{\mathfrak g}, l\varepsilon_1- k\varepsilon_{n+1} \rangle_0 = 2l(n+l)+2k(n+k). 
\end{align*}
We conclude from \eqref{eq2:lambda-pi-tau-L} that the corresponding eigenvalue is
\begin{equation}\label{eq:prelim_lambdapq}
\begin{aligned}
\lambda^{\pi_{k,l}, 1_\H\otimes \phi_{l-k}} (r,s)
&= (r^2-s^2)\, \lambda^{1_\H\otimes \phi_{k-l} }  + s^2 \lambda^{\pi_{k,l}}
\\
&= 2(k-l)^2(r^2-s^2)  +  \big(2l(n+l)+2k(n+k) \big) s^2
\\
&=\big( 4kl+2n(k+l) \big) s^2 +2(k-l)^2 r^{2}.
\end{aligned}
\end{equation}

For convenience of notation, let us reindex $(k,l)\in\N_0^2$ as $(p,q)\in\N_0^2$,
\begin{equation}\label{eq:reindex}
p := \min\{k,l\}, \qquad q:=\max\{k,l\}-\min\{k,l\}  =|k-l|.
\end{equation}
Since $kl=p(p+q)$, $k+l=2p+q$, and $(k-l)^2=q^2$, \eqref{eq:prelim_lambdapq} is equal to \eqref{eqSU:lambda}. 
Moreover, $q=0$ if and only if $k=l$, which is equivalent to $\phi_{l-k}=1_\LL$, proving the claim regarding basic eigenvalues.

We conclude by determining the contribution $m_{p,q}$ to the multiplicity of the eigenvalue $\lambda^{(p,q)}(r,s)\in\spec\!\left(\Ss^{2n+1},\g_{(r,s)}\right)$.
On the one hand, if $q=0$, the only solution to \eqref{eq:reindex} is $(k,l)=(p,p)$, so this contribution is $m_{p,0}=\dim V_{p,p}$. On the other hand, if $q>0$, then there are \emph{two} solutions to \eqref{eq:reindex}, namely $(k,l)=(p+q,p)$ and $(k,l)=(p,p+q)$, yielding a contribution of $m_{p,q}= \dim V_{p+q,p}+ \dim V_{p,p+q}=2\dim V_{p+q,p}$.
Therefore, \eqref{eqSU:d_pq} now follows from \eqref{eq:SUdimV_prelim}.
\end{proof}

\section{\texorpdfstring{Eigenvalues of the Laplacian on $\Ss^{4n+3}$}{Eigenvalues of the Laplacian}}\label{sec:fullspectra-S^4n+3}

This short section gives the full Laplace spectrum of the homogeneous spheres $\big(\Ss^{4n+3},\mathbf h(t)\big)$, $n\geq1$, based on \cite{blp-firsteigenvalue}. Following the same notation as above, we set
\begin{equation}\label{eqSp:GKH}
\begin{aligned}
\G&= \Sp(n+1), &\LL&= \left\{ \begin{pmatrix} I & 0 \\ 0 & z\end{pmatrix}: z\in\Sp(1)\right\},\\
\H&= \left\{ \begin{pmatrix} A & 0 \\ 0 & 1\end{pmatrix}: A\in\Sp(n)\right\}, &
\K& 
=\left\{ \begin{pmatrix} A & 0 \\ 0 & z\end{pmatrix}: A\in\Sp(n),\, z\in\Sp(1)\right\}.
\end{aligned}
\end{equation}
Clearly, $\H\simeq \Sp(n)$, $\LL\simeq \Sp(1)$, and $\K=\H\LL\simeq \Sp(n)\Sp(1)$, as in Subsection~\ref{subsec:prod}, and it is well known that $\G/\H\cong \Ss^{4n+3}$ and $\G/\K\cong \Hr P^n$. 
It is easy to check that
\begin{equation}\label{eqSp:p-q}
	\begin{aligned}
		\fp& =\left\{ \begin{pmatrix} 0 & 0\\ 0 &a \end{pmatrix}: a\in \op{Im} \Hr \right\}, &
		\fq& =\left\{ \begin{pmatrix} 0 & v \\ -  v^* & 0\end{pmatrix}: v\in \Hr^n \right\}.
	\end{aligned}
\end{equation}
They satisfy $\fk=\fh\oplus\fp$ and $\fg=\fk\oplus \fq= \fh\oplus (\fp\oplus\fq)$. 
As subrepresentations of the isotropy representation of $\H$, $\fp$ is equivalent to three copies of the trivial representation, and $\fq$ is the standard representation.

Consider again the $\G$-invariant metrics $\g_{(r,s)}$ on $\G/\H$, as in Subsection~\ref{subsec:basicsetting}. This is a $2$-parameter subfamily of the $4$-parameter family of $\G$-invariant metrics on $\G/\H$, see e.g.~\cite[Sec.~3.2]{blp-firsteigenvalue}. For all $t>0$, the metric $\mathbf h(t)$ in \eqref{eq:hopfbundles} is isometric to $\g_{(\frac{1}{t},1)}$.

\begin{proposition}\label{thmSp:spec}
For all $n\geq1$, the spectrum of the Laplace--Beltrami operator on $\big(\Ss^{4n+3},\g_{(r,s)}\big)$ consists of eigenvalues
\begin{equation}
\label{eqSp:lambda}
\lambda^{(p,q)}(r,s) = \big(4p(p+q+2n+1)+4qn\big)s^2 +q(q+2)r^2, \quad p,q\in\N_0,
\end{equation}
which are basic if $q=0$, and have multiplicity 
\begin{equation}
\label{eqSp:d_pq}
m_{p,q} = \frac{(2p+q+2n+1)(q+1)^2}{(2n+1)(p+q+1)}\binom{p+q+2n}{p+q} \binom{p+2n-1}{p}.
\end{equation}
\end{proposition}

\begin{proof}
This follows from Corollary~\ref{cor:specHL}, analogously to Proposition~\ref{thmSU:spec}, using the appropriate branching law. Alternatively, it follows from \cite[Lem.~3.2, Rem.~3.3]{blp-firsteigenvalue} replacing $(p,q)$ with $(p+q,p)$, and setting $a=b=c=r/\sqrt2$, which implies that $\nu_j^{(q)}(a,b,c)=\tfrac12r^2q(q+2)$ for all $1\leq j\leq q+1$.
Accordingly, the multiplicity \eqref{eqSp:d_pq} is $q+1$ times that in \cite[(3.11)]{blp-firsteigenvalue}, since $\nu_j^{(q)}$ does not depend on $j$ in this case.
\end{proof}

\section{\texorpdfstring{Eigenvalues of the Laplacian on $\Ss^{15}$}{Eigenvalues of the Laplacian}}\label{sec:fullspectra-S^15}

In this section, we determine the full Laplace spectrum of $\left(\Ss^{15},\mathbf k(t)\right)$. 
The Lie groups $\H\subset \K\subset \G$ in this case do not follow the pattern \eqref{eqSU:GKH} and \eqref{eqSp:GKH} of the previous sections. Namely, the inclusion $\H\subset\K$ is \emph{not} given by a block embedding, and there is \emph{no} Lie subgroup $\LL\subset\K$ such that $\K=\H\LL$. In particular, Corollary~\ref{cor:specHL} no longer applies.

Let $\G=\Spin(9)$, and identify its Lie algebra $\fg=\spin(9)$ with $\so(9)$ in the standard way.
Let $\K$ be the subgroup of $\G$ isomorphic to $\Spin(8)$ with Lie algebra
\begin{equation*}
\fk = \{\diag(X,0)\in\fg: X\in\so(8)\}\simeq \so(8). 
\end{equation*}
Clearly, $\G/\K\cong \Ss^{8}$. In order to define the subgroup $\H\subset \K$, which is isomorphic to $\Spin(7)$, but whose Lie algebra $\fh\subset\fk$ is not a block inclusion as $\fk \subset \fg$ above, we follow an approach tailored to apply the branching law of Baldoni-Silva~\cite[\S6]{Baldoni79}.

Let $\F\cong \F_4^{-20}$ be the simply connected Lie group associated to the real simple Lie algebra $\ff$ of type FII. The maximal compact subgroup of $\F$ (which is unique up to conjugation) is isomorphic to $\G$, and $\F/\G\cong \Ca H^2$.
Fix the maximal torus $\T\subset \G\subset \F$ with Lie algebra 
\begin{equation}\label{eqSpin(9):cartansubalgebra}
	\ft = 
	\left\{ \diag\left( 
	\begin{pmatrix} 0 & \mi\theta_1 \\ -\mi\theta_1&0 \end{pmatrix}
	,\dots,
	\begin{pmatrix} 0 & \mi\theta_4\\ -\mi\theta_4&0 \end{pmatrix}
	,1
	\right) \in\fg: \theta_1,\dots,\theta_4\in\R\right\}.
\end{equation}
Its complexification $\ft_\C$ is a Cartan subalgebra of $\fg_\C$, with elements as in \eqref{eqSpin(9):cartansubalgebra} where $\theta_1,\dots,\theta_4\in\C$. 
The functionals $\varepsilon_j\colon \mathfrak t_\C^*\to \C$ that map such an element to $\theta_j$ form a $\C$-basis of $\mathfrak t_\C^*$.
Fix an order on $\mi\ft$ such that the corresponding positive root systems of $\fg_\C$ and $\ff_\C$ with respect to $\ft_\C$ are 
\begin{equation*}
    \begin{aligned}
        \Phi^+(\fg_\C,\ft_\C) &=\{\ee_i:1\leq i\leq 4\} \cup \{\ee_i\pm\ee_j:1\leq i<j\leq 4\},\\
\Phi^+(\ff_\C,\ft_\C) &= \Phi^+(\fg_\C,\ft_\C) \cup\left\{\tfrac12(\ee_1 \pm\ee_2 \pm\ee_3 \pm\ee_4)\right\}.
    \end{aligned}
\end{equation*}
Let $\fm$ be the orthogonal complement of $\fg$ on $\ff$ with respect to the Killing form of $\ff$, so that $\ff=\fg\oplus \fm$ is a Cartan decomposition. 
Set $\alpha=\tfrac12(\ee_1-\ee_2-\ee_3-\ee_4) \in \Phi^+(\ff_\C,\ft_\C)$, and choose root vectors $X_{\alpha}\in (\ff_\C)_{\alpha}$ and $X_{-\alpha}\in (\ff_\C)_{-\alpha}$ satisfying $[X_{\alpha},X_{-\alpha}] \in\fp$. 
Then $\fa:=\R(X_{\alpha_4} + X_{-\alpha_4})$ is a maximal abelian subalgebra of $\fp$. Finally, define
$\H$ as the centralizer of $\fa$ in $\K$, that is, 
\begin{equation*}
\H=\{k\in\K : \Ad(k)\cdot \fa=0\}.
\end{equation*} 
It can be checked that its Lie algebra $\fh = \{X\in\fk: [X,\fa]=0\}$ is isomorphic to $\so(7)$, and $\H\simeq\Spin(7)$, $\G/\H\cong\Ss^{15}$, and $\K/\H\cong \Ss^{7}$.

Consider the $\innerdots_0$-orthogonal complements $\fp$ and $\fq$, and $\G$-invariant metrics $\g_{(r,s)}$, as in Subsection~\ref{subsec:basicsetting}. As subrepresentations of the isotropy representation of $\H$, $\fp$ is the standard representation, and $\fq$ is the spin representation.
Since they are irreducible and non-equivalent, every $\G$-invariant metric on $\G/\H$ is isometric to some $\g_{(r,s)}$; e.g., for all $t>0$, the metric ${\bf k}(t)$ in \eqref{eq:hopfbundles} is isometric to $\g_{(\frac{1}{t},2)}$.
 
\begin{proposition}\label{thmSpin(9):spec}
The spectrum of the Laplace--Beltrami operator on $\big(\Ss^{15},\g_{(r,s)}\big)$ consists of eigenvalues
\begin{equation}\label{eqSpin(9):lambda}
\lambda^{(p,q)}(r,s) =\big( p^2+p(q+7)+2q \big)s^2 + q(q+6)r^2, \quad p,q\in\N_0,
\end{equation}
which are basic if $q=0$, and have multiplicity
\begin{equation}\label{eqSpin(9):d_kl}
m_{p,q} = \frac{2p+q+7}{7} \left(1+\frac{q}{3}\right) \frac{\binom{p+q+6}{p+q} \binom{p+3}{p} \binom{q+5}{q}}{\binom{p+q+3}{p+q}}.
\end{equation}
\end{proposition}

\begin{proof}
In order to apply Theorem~\ref{thm2:spec}, we first state the branching law from $\G$ to $\H$ in order to determine $\widehat \G_{\H}$ and the integers $[1_\H:\tau|_\H]$ and $[\tau:\pi|_{\K}]$ for $\tau\in\widehat \K$ satisfying $[\tau:\pi|_{\K}]>0$. 
Note that $\T\cap \K$ is a maximal torus in $\K$, and the corresponding positive root system is $\Phi^+(\fk_\C, (\ft\cap\fk)_\C)= \{\ee_i\pm\ee_j:1\leq i<j\leq 4\}$ with simple roots $\ee_1-\ee_2, \ee_2-\ee_3, \ee_3-\ee_4, \ee_3+\ee_4$. 
We have that 
\begin{equation*}
\begin{aligned}
P^+(\G) &= \left\{ \sum_{i=1}^4 a_i\ee_i: 
\begin{array}{l}
a_1\geq a_2\geq a_3\geq a_4\geq0,\\
2a_i\in\Z,\,  a_i-a_j\in\Z \text{ for }  1\leq i,j\leq 4 
\end{array}
 \right\},
 \\
P^+(\K) &= \left\{ \sum_{i=1}^4 a_i\ee_i: 
\begin{array}{l}
a_1\geq a_2\geq a_3\geq |a_4|,\\
2a_i\in\Z,\,  a_i-a_j\in\Z \text{ for }  1\leq i,j\leq 4 
\end{array}
\right\}.
\end{aligned}
\end{equation*}
The fundamental weights of $\Phi^+(\fg_\C,\ft_\C)$ are $\omega_1=\ee_1$, $\omega_2=\ee_1+\ee_2$, $\omega_3=\ee_1+\ee_2+\ee_3$, and $\omega_4= \tfrac12(\ee_1+\ee_2+\ee_3+\ee_4)$, and they satisfy $P^+(\G)=\bigoplus_{i=1}^4\N_0 \omega_i$. 

We define $\varphi\colon \ft_\C^* \to\ft_\C^*$ to be the linear map determined by 
\begin{equation*}
\begin{aligned}
\varphi(\ee_1)&= \tfrac12(+\ee_1+\ee_2+\ee_3-\ee_4),\\
\varphi(\ee_2)&= \tfrac12(+\ee_1+\ee_2-\ee_3+\ee_4),\\
\varphi(\ee_3)&= \tfrac12(+\ee_1-\ee_2+\ee_3+\ee_4),\\
\varphi(\ee_4)&= \tfrac12(-\ee_1+\ee_2+\ee_3+\ee_4).
\end{aligned}
\end{equation*}
One can check that $\varphi^2=\Id$. 
Since $\varphi$ permutes the simple roots of $\Phi^+(\fk_\C, (\ft\cap\fk)_\C)$, namely, 
$\varphi(\ee_1-\ee_2) = \ee_3-\ee_4$,
$\varphi(\ee_2-\ee_3) = \ee_2-\ee_3$, 
$\varphi(\ee_3-\ee_4) = \ee_1-\ee_2$, and 
$\varphi(\ee_3+\ee_4) = \ee_3+\ee_4$, 
we have that $\varphi$ is an automorphism of $\Phi^+(\fk_\C, (\ft\cap\fk)_\C)$, which extends to an automorphism of $\fk_\C$ that we denote again by $\varphi$. 
It turns out (see \cite[p.~248]{Baldoni79}) that $\varphi(\fh)$ is a copy of $\so(7)$ embedded in $\fk\simeq \so(8)$.
More precisely, $\varphi(\fh) = \{\diag(X,0,0)\in\fg: X\in\so(7)\}$, so the simple roots are:
$
	\ee_1-\ee_2=\varphi(\ee_3-\ee_4),\,
	\ee_2-\ee_3=\varphi(\ee_2-\ee_3), \,
	\ee_3= \varphi(\tfrac12(+\ee_1-\ee_2+\ee_3+\ee_4)),
$ 
and $\Phi^+(\varphi(\fh)_\C, (\ft\cap\varphi(\fh))_\C) = \{\ee_i:1\leq i\leq 3\} \cup \{\ee_i\pm\ee_j:1\leq i<j\leq 3\}$, hence
\begin{equation*}
P^+(\H') = \left\{ \sum_{i=1}^3 c_i\ee_i: 
\begin{array}{l}
c_1\geq c_2\geq c_3\geq0,\\
2c_i\in\Z,\,  c_i-c_j\in\Z \text{ for }  1\leq i,j\leq 3
\end{array}
\right\},
\end{equation*}
where $\H'$ denotes the only connected Lie subgroup of $\K$ with Lie algebra $\varphi(\fh)$. 

We are now in position to state the branching law from $\G$ to $\H$ established by Baldoni-Silva~\cite[Thm.~6.3]{Baldoni79}. For $\Lambda=\sum_{i=1}^4 a_i\ee_i\in P^+(\G)$,
\begin{equation}\label{eqSpin(9):branching}
\pi_\Lambda|_{\H} = 
\sum_{\substack{
	b_1\ee_1+b_2\ee_2+b_3\ee_3+b_4\ee_4 \in P^+(\K):\\
	a_1\geq b_1\geq a_2\geq b_2\geq a_3\geq b_3\geq a_4\geq |b_4|,\\
	a_1-b_1\in\Z,\\
	b_1'\ee_1+\dots+b_4'\ee_4:=\varphi(b_1\ee_1+\dots+b_4\ee_4),
}} \quad 
\sum_{\substack{
	\nu:= c_1\ee_1 +c_2\ee_2+ c_3\ee_3 \in P^+(\H'):\\
	b_1'\geq c_1\geq b_2'\geq c_2\geq b_3' \geq c_3\geq |b_4'|,\\
	b_1'-c_1\in\Z
}} \quad
\sigma_{\nu}\circ\varphi.
\end{equation}

We claim that 
\begin{equation*}
\widehat \G_\H = \{\pi_{p,q}:=\pi_{p\omega_1+q\omega_4}: p,q\in\N_0 \}.
\end{equation*}
Recall that $\omega_1=\ee_1$ and $\omega_4= \tfrac12(\ee_1+\ee_2+\ee_3+\ee_4)$.
Let $\Lambda=\sum_{i=1}^4a_i\ee_i\in P^+(\G)$. 
Of course, the trivial  $\H$-representation $1_\H$ coincides with $\sigma_\nu$ with $\nu=0$, i.e.\ $c_1=c_2=c_3=0$. 
Therefore, if $[1_\H:\pi_{\Lambda}|_{\H}]>0$, then the coefficients in \eqref{eqSpin(9):branching} satisfy $b_2'=b_3'=b_4'=0$ and $b_1'\in\N_0$, which gives $b_1=b_2=b_3=-b_4=b_1'/2$, and consequently $a_2=a_3=a_4=b_1'/2$ and $a_1-b_1'/2\in\N_0$.
We conclude that $\Lambda= {b_1'}\omega_4+(a_1-b_1'/2)\omega_1$, as claimed.

Moreover, for $\Lambda=p\omega_1+q\omega_4$ with $p,q\in\N_0$, we have that $[1_\H:\pi_\Lambda|_\H]=1$ (i.e.\ $1_\H$ occurs exactly once in $\pi_\Lambda|_\H$) since the coefficients $b_i,b_i'$ for $1\leq i\leq 4$ in \eqref{eqSpin(9):branching} are uniquely determined in terms of $p$ and $q$; indeed, $b_1'=2b_1=2b_2=2b_3=-2b_4=q$ and $b_2'=b_3'=b_4'=0$.
This implies  there exists only one $\mu\in P^+(\K)$ satisfying $[1_\H:\tau_\mu|_\H]= [\tau_\mu:\pi_{\Lambda}|_\K]=1$, which is given by $\mu_q :=\sum_{i=1}^4b_i\ee_i= \frac{q}{2}(\ee_1+\ee_2+\ee_3-\ee_4)$. 

By the above and Theorem~\ref{thm2:spec}, the eigenvalues of $(\Ss^{15}, \g_{(r,s)})$ are 
\begin{equation*}
\lambda^{(p,q)} (r,s)= \lambda^{\pi_{p,q},\tau_q}(r,s) = (r^2-s^2)\, \lambda^{\tau_q}  + s^2 \lambda^{\pi_{p,q}}, \quad p,q\in\N_0,
\end{equation*}
where $\tau_q\in\widehat\K$ has highest weight $\mu_q= {\tfrac q2(\ee_1+\ee_2+\ee_3-\ee_4)}$, with multiplicity $m_{p,q}=\dim V_{\pi_{p,q}}$ equal to \eqref{eqSpin(9):d_kl} by the Weyl dimension formula, see e.g.~\cite[Thm.~5.84]{Knapp-book-beyond}. 
Moreover, $\widehat \G_\K=\{\pi_{p,0}=\pi_{p\ee_1}:p\in\N_0\}$ by the classical branching law from $\Spin(9)$ to $\Spin(8)$, so $\lambda^{(p,q)} (r,s)$ is basic if $q=0$.

The only remaining step is to determine the scalars $\lambda^{\tau_q}$ and $\lambda^{\pi_{p,q}}$. It is easy to check that $\langle \varepsilon_i, \varepsilon_j\rangle_0 =\delta_{ij}$ for all $1\leq i,j\leq 4$. 
Freudenthal's formula \eqref{eq:Casimirscalar} gives
\begin{equation*}
\begin{aligned}
\lambda^{\tau_q} 
&= \inner{\mu_q}{\mu_q+2\rho_{\fk}} \\
&= \left\langle \tfrac{q}{2} (\ee_1+\ee_2+\ee_3-\ee_4) , \tfrac{q}{2} (\ee_1+\ee_2+\ee_3-\ee_4) + \textstyle \sum\limits_{i=1}^{4} (8-2i)\ee_i \right\rangle\\
&= \tfrac{q}{2} (\tfrac{q}{2}+6) + \tfrac{q}{2} (\tfrac{q}{2}+4) + \tfrac{q}{2} (\tfrac{q}{2}+2) + \tfrac{q^2}{4}\\
&=q(q+6).
\end{aligned}
\end{equation*}
Similarly, since $2\rho_{\fg}= \sum_{i=1}^{4} (9-2i)\ee_i$, we have that
\begin{equation*}
\begin{aligned}
\lambda^{\pi_{p,q}} 
&= \inner{p\omega_1+ q\omega_4}{p\omega_1+k\omega_4 +2\rho_{\fg}} 
\\
&= (p+\tfrac{q}{2}) (p+\tfrac{q}{2}+7) + \tfrac{q}{2} (\tfrac{q}{2}+5) + \tfrac{q}{2} (\tfrac{q}{2}+3) + \tfrac{q}{2}(\tfrac{q}{2}+1)
\\
&= p^2+p(q+7)+q^2+8q. 
\end{aligned}
\end{equation*}
Combining the above, one obtains \eqref{eqSpin(9):lambda}, which concludes the proof.
\end{proof}

\section{Unified formulae}\label{sec:unified}

In order to prove Theorem~\ref{mainthm:A} in the Introduction, we collect in Table~\ref{tab:eigenvalues} the Laplace spectra of the homogeneous spheres $\big(\Ss^{N-1},\g(t)\big)$ in \eqref{eq:universal-hopf}, as computed in Propositions~\ref{thmSU:spec}, \ref{thmSp:spec}, and \ref{thmSpin(9):spec}, keeping in mind the isometries relating homogeneous metrics in their geometric description \eqref{eq:hopfbundles} with their algebraic description \eqref{eq:innprodrs}.

\begin{table}[!ht]
\begin{tabular}{|c|c|l|}
\hline
$\mathds K$ & $\;\big(\Ss^{N-1},\g(t)\big)$ \rule[-2ex]{0pt}{0pt} \rule{0pt}{3ex} & Parameters $(r,s)$ and Laplace--Beltrami spectrum
\\
\hline
$\begin{array}{c} \C\\[4pt] d=1 \end{array}$ & $\big(\Ss^{2n+1},\mathbf g(t)\big)\!\!\!$\rule[-6ex]{0pt}{0pt} \rule{0pt}{7ex} &
$\begin{aligned}
(r,s)&=\big(\tfrac{1}{t\sqrt2},1\big)\\
\lambda^{(p,q)}(t) &= 4p(p+q+n) + 2nq+ q^2\tfrac{1}{t^2} \\
m_{p,q} &= (2-\delta_{q0}) \tfrac{2p+q+n}{n} \textstyle\binom{p+q+n-1}{p+q} \binom{p+n-1}{p}
\end{aligned}$ \\[15pt]
\hline
$\begin{array}{c} \Hr\\[4pt] d=2 \end{array}$& $\big(\Ss^{4n+3},\mathbf h(t)\big)\!\!\!$\rule[-6ex]{0pt}{0pt} \rule{0pt}{7ex} & 
$\begin{aligned}
(r,s)&=\big(\tfrac{1}{t},1\big)\\
\lambda^{(p,q)}(t) &= 4p(p+q+2n+1)+4nq+q(q+2)\tfrac{1}{t^2} \\
m_{p,q} &= \tfrac{(2p+q+2n+1)(q+1)^{2}}{(2n+1)(p+q+1)}\textstyle\binom{p+q+2n}{p+q} \binom{p+2n-1}{p}
\end{aligned}$ \\[15pt]
\hline
$\begin{array}{c} \Ca\\[4pt] d=4 \end{array}$ & $\big(\Ss^{15},\mathbf k(t)\big)\!\!\!$\rule[-6ex]{0pt}{0pt} \rule{0pt}{7ex}&
$\begin{aligned}
(r,s)&=\big(\tfrac{1}{t},2\big)\\
\lambda^{(p,q)}(t) &= 4p(p+q+7)+8q+q(q+6)\tfrac{1}{t^2} \\
m_{p,q} &=  \textstyle \frac{2p+q+7}{7} (1+\tfrac{q}{3}) \binom{p+q+6}{p+q} \binom{p+3}{p} \binom{q+5}{q}/ \binom{p+q+3}{p+q}
\end{aligned}$ \\[15pt]
\hline
\end{tabular}

\caption{Eigenvalues of the homogeneous spheres $\big(\Ss^{N-1},\g(t)\big)$ in \eqref{eq:universal-hopf}, where $N=2d(n+1)$, $d=\dim_\C \mathds K\in\{1,2,4\}$, and $p,q\in\N_0$.}
\label{tab:eigenvalues}
\end{table}

\begin{proof}[Proof of Theorem~\ref{mainthm:A}]
Replacing $d\in\{1,2,4\}$ in \eqref{eqthm:lambdapq-universal} and \eqref{eqthm:dpq-universal}, one obtains $\lambda^{(p,q)}(t)$ and $m_{p,q}$ as listed in Table~\ref{tab:eigenvalues}. By Propositions~\ref{thmSU:spec}, \ref{thmSp:spec}, and~\ref{thmSpin(9):spec}, these are the eigenvalues and respective multiplicities of the Laplace--Beltrami operator on the corresponding sphere $\big(\Ss^{N-1},\g(t)\big)$, and $\lambda^{(p,q)}(t)$ is basic if $q=0$.
\end{proof}

Recall that the distance sphere $S(r)\subset M$ is isometric to $\big(\Ss^{N-1},\alpha^2\,\g(t)\big)$, where $(\alpha,t)$ is $(\sin r,\cos r)$ or $(\sinh r,\cosh r)$ according to $M=\Kr P^{n+1}$ or $M=\Kr H^{n+1}$, cf.~\eqref{eq:t&alpha}, and $N=\dim M=2d(n+1)$.
Rescaling all spaces in the Riemannian submersion \eqref{eq:universal-hopf} by $\alpha$, since its fibers are totally geodesic, one obtains the inclusions of spectra
\begin{equation}\label{eq:inclusion1}
\tfrac{1}{\alpha^2}\Spec\big(\Kr P^n\big) \subset  \Spec\big(S(r)\big) \subset \tfrac{1}{\alpha^2}\Big(\Spec\big(\Kr P^n\big)+\Spec\big(\Ss^{2d-1}_{t}\big)\Big),
\end{equation}
where $+$ is the Minkowski sum of sets, $A+B=\{a+b:a\in A,\, b\in B\}$. These inclusions are also immediate from Theorem~\ref{mainthm:A}, by analyzing the case $q=0$ in \eqref{eqthm:lambdapq-universal}.

However, there is another remarkable inclusion of spectra, given by the following:

\begin{corollary}
The Laplace--Beltrami spectrum of $S(r)\subset M$ satisfies
\begin{equation}\label{eq:inclusion2}
\Spec\big(S(r)\big) \subset \Spec\big(\Ss^{N-1}_{\alpha}\big)\pm \Spec\big(\Ss^{2d-1}_{t}\big),
\end{equation}
where $+$ is used if $M$ is projective, and $-$ if $M$ is hyperbolic.
\end{corollary}

Let us first prove \eqref{eq:inclusion2} with a geometric argument assuming that $\Kr\in\{\C,\Hr\}$ and $M=\Kr P^{n+1}$ is a projective space, hence the base of the Riemannian submersion 
\begin{equation}\label{eq:bigbundle}
\Ss^{2d-1}_1\longrightarrow \Ss^{N+2d-1}_1\longrightarrow \Kr P^{n+1}
\end{equation}
whose totally geodesic fibers are precisely the orbits of the 
free action of the group $\Ss^{2d-1}_1\subset \Kr^*$ of multiplicative units on the unit sphere
\begin{equation}\label{eq:unitsphereK}
\Ss^{N+2d-1}_1\subset\R^{N+2d} \cong \mathds K^{N/2d+1}.
\end{equation}
Being a distance sphere, the preimage of $S(r)\subset \Kr P^{n+1}$ under this submersion is the boundary of the tubular neighborhood of radius $r$ of the fiber that corresponds to the central point of $S(r)$. Since this fiber is an orbit of the aforementioned action on \eqref{eq:unitsphereK}, this boundary is a product of spheres, isometric to
\begin{equation*}
\Ss^{N-1}_\alpha\times\Ss^{2d-1}_t = \Ss^{N+2d-1}_1 \cap (\R^{N}\oplus\R^{2d}) \cong \Ss^{N+2d}_1\cap (\mathds K^{N/2d}\oplus\Kr),
\end{equation*}
which proves \eqref{eq:inclusion2} for $\Kr\in\{\C,\Hr\}$ and $M$ projective.
The same argument can be generalized to the case in which $M=\Kr H^{n+1}$ is hyperbolic, interpreting \eqref{eq:unitsphereK} as the unit pseudo-sphere in the pseudo-Riemannian vector space $\Kr^{N/2d}\oplus\Kr$ of signature $(N,2d)$, analogous to the discussion in \cite[Sec.~6]{docarmo88}.

Nevertheless, the above arguments \emph{do not} apply to $\Kr=\Ca$ in either case because it is not associative; in particular, its unit sphere $\Ss^7_1$ is not a group. Moreover, it is well known that there are no fiber bundles $\Ss^\ell \to \Ca P^2$ such as \eqref{eq:bigbundle} for topological reasons \cite{browder}.
Thus, it is a somewhat surprising consequence of Theorem~\ref{mainthm:A} that \eqref{eq:inclusion2} still holds for $\Kr=\Ca$, in both projective and hyperbolic cases. In fact, \eqref{eq:inclusion2} can be explicitly parametrized, for all $\Kr\in\{\C,\Hr,\Ca\}$ at once, using that, by \eqref{eqthm:lambdapq-universal},
\begin{equation*}
\lambda^{(p,q)}(t)=(2p+q)(2p+q+N-2)+q(q+2d-2)\left(\frac{1}{t^2}-1\right),
\end{equation*}
and, by \eqref{eq:t&alpha}, we have $\pm\frac{\alpha^2}{t^2}=\left(\frac{1}{t^2}-1\right)$ according to $M=\Kr P^{n+1}$ or $M=\Kr H^{n+1}$. 
\section{Resonance and rigidity of distance spheres}\label{sec:bif}

In this section, we recall the variational and bifurcation framework for constant mean curvature (CMC) hypersurfaces and prove Theorem~\ref{mainthm:B} in the Introduction.

\subsection{CMC spheres}
Given an $N$-dimensional Riemannian manifold $(M,\g)$, let $\Emb(\Ss^{N-1},M)$ be the space of $C^{2,\alpha}$ unparametrized embeddings of $\Ss^{N-1}$ into $M$, i.e., equivalence classes of embeddings $\x\colon \Ss^{N-1}\to M$ for the action of $\Diff(\Ss^{N-1})$ by right-composition. Consider the family of functionals
\begin{equation}\label{eq:fH}
\begin{aligned}
f_H&\colon \Emb(\Ss^{N-1},M)\longrightarrow\R\\
f_H(\x)&=\Area(\x)+H\,\Vol(\x),
\end{aligned}
\end{equation}
where $\Area(\x)$ denotes the $(N-1)$-volume of $\x(\Ss^{N-1})$, and $\Vol(\x)$ the $N$-volume of the region enclosed by $\x(\Ss^{N-1})$. It is well known that critical points of \eqref{eq:fH} are precisely the embedded spheres in $M$ with constant mean curvature $H$. Moreover, the second variation of \eqref{eq:fH} at a critical point is represented by the Jacobi operator
\begin{equation}\label{eq:jacobi}
J_\x(\phi)=\Delta_{\x} \phi - (\Ric(\vec n_\x)+\|A_{\x}\|^2)\phi,
\end{equation}
acting on the space of functions $\phi\colon \Ss^{n-1}\to\R$ with $\int_{\Ss^{N-1}}\phi=0$,
where $\Delta_\x$ is the Laplace--Beltrami operator on $\Ss^{N-1}$ with respect to the metric induced by the embedding $\x\colon \Ss^{N-1}\to M$, $\vec n_\x$ is a unit normal vector field to $\x(\Ss^{N-1})\subset M$, and $\|A_\x\|$ is the Hilbert--Schmidt norm of its second fundamental form; for details, see e.g.~\cite[Sec.~2]{docarmo88}. 
Functions $\phi\in\ker J_\x$ are called \emph{Jacobi fields}, and the number $\imorse(\x)$ of negative eigenvalues of $J_\x$, counted with multiplicity, is called the \emph{Morse index} of~$\x$.
Moreover, $\x$ is \emph{stable} if and only if $J_\x$ is positive-semidefinite, i.e., $\imorse(\x)=0$, and \emph{nondegenerate} if and only if $\ker J_\x=\{0\}$. 

\subsection{Equivariant rigidity and resonance}
If a Lie group $\G$ acts isometrically on $M$, then \eqref{eq:fH} is clearly invariant under left-composition with this action, so the entire $\G$-orbit of a critical point is critical. Moreover, since \eqref{eq:jacobi} is $\G$-equivariant, each Killing field $X\in\fg$ determines a Jacobi field $\phi_X=\langle X,\vec n_\x\rangle\in\ker J_\x$. In this context, we say $\x$ is \emph{$\G$-equivariantly nondegenerate} if $\ker J_\x$ consists solely of such Jacobi fields induced by the $\G$-action.

Let $\K$ be the $\G$-isotropy of $x_0\in M$, and assume the $\K$-action is transitive on all geodesic distance spheres $S(r)\subset M$ centered at $x_0$. In particular, the (unparametrized) embeddings 
\begin{equation}\label{eq:xr}
\x_r\colon \Ss^{N-1}\to M, \quad \x_r(\Ss^{N-1})=S(r),
\end{equation}
have constant mean curvature $H(S(r))$ for each $r$. Furthermore, assume that the map $r\mapsto H(S(r))$ is a diffeomorphism, so that $\x_r$ may also be parametrized by its mean curvature. 
In this context, an appropriate $\G$-equivariant version of the Implicit Function Theorem~\cite[Thm.~1.4]{g-implicit} implies:

\begin{theorem}\label{thm:IFT}
Suppose \eqref{eq:xr} is $\G$-equivariantly nondegenerate if $r=r_*$. There exists $\varepsilon>0$ such that, if an embedded sphere $\Sigma\subset M$ has constant mean curvature $H(\Sigma)=H(S(r))$, $r\in(r_*-\varepsilon,r_*+\varepsilon)$ and, up to isometries in $\G$, is sufficiently close to $S(r)$ in $C^{2,\alpha}$-topology, then $\Sigma$ is congruent to $S(r)$ via an isometry in~$\G$.
\end{theorem}

The radii $r_*$ for which the conclusion of Theorem~\ref{thm:IFT} fails are called \emph{resonant}:

\begin{definition}
We say $r_*$ is a \emph{resonant radius} if there exist sequences $r_j$ of radii converging to $r_*$, and $\Sigma_j\subset M$ of embedded CMC spheres converging to $S(r_*)$ in $C^{2,\alpha}$-topology, such that $H(\Sigma_j)=H(S(r_j))$ for all $j$, and $\Sigma_j$ is not congruent to $S(r_j)$ via any isometry in $\G$.
\end{definition}

Clearly, by Theorem~\ref{thm:IFT}, a necessary condition for $r_*$ to be resonant is that $\x_{r_*}$ is \emph{not} $\G$-equivariantly nondegenerate. The following sufficient condition for resonancy is a direct consequence of the equivariant bifurcation criterion \cite[Thm.~5.4]{g-bifurcation}.

\begin{theorem}\label{thm:g-bif}
If for all $\varepsilon>0$ sufficiently small, $\x_{r_*-\varepsilon}$ and $\x_{r_*+\varepsilon}$ are $\G$-equivari\-antly nondegenerate and $\imorse(\x_{r_*-\varepsilon})\neq\imorse(\x_{r_*+\varepsilon})$, then $r_*$ is resonant.
\end{theorem}

\subsection{Rank one symmetric spaces}
We now briefly revisit some well-known aspects of the geometry of rank one symmetric spaces that are used in the proof of Theorem~\ref{mainthm:B}. First, recall that the symmetric pairs $(\G,\K)$ that give rise to such spaces $M=\G/\K$ are as listed in Table~\ref{tab:symmpairs}. 
\begin{table}[!ht]
	\begin{tabular}{|c|c|c|}
		\hline
		$\G/\K$ & $\G$ \rule[-1.5ex]{0pt}{0pt} \rule{0pt}{3ex} & $\K$\\
		\hline\hline
		$\C P^{n+1}$\rule{0pt}{3ex} & $\SU(n+2)$ & $\mathsf{S}(\U(n+1)\U(1))$\\
		$\Hr P^{n+1}$ & $\Sp(n+2)$ & $\Sp(n+1)\Sp(1)$\\
		$\Ca P^{2}$\rule[-1ex]{0pt}{0pt} & $\mathsf F_4$ & $\Spin(9)$\\
		\hline
$\C H^{n+1}$ \rule{0pt}{3ex} & $\SU(n+1,1)$ & $\mathsf{S}(\U(n+1)\U(1))$\\
$\Hr H^{n+1}$ & $\Sp(n+1,1)$ & $\Sp(n+1)\Sp(1)$\\
$\Ca H^{2}$\rule[-1ex]{0pt}{0pt} & $\mathsf F_4^{-20}$ & $\Spin(9)$\\
\hline
\end{tabular}
\caption{Symmetric pairs $(\G,\K)$ of rank one corresponding to the projective spaces $\Kr P^{n+1}$, and their noncompact duals, the hyperbolic spaces $\Kr H^{n+1}$.}\label{tab:symmpairs}
\end{table}
These semisimple Lie groups $\G$ act transitively on $M$, and $\K\subset \G$ is identified with the isotropy of a point $x_0\in M$, so its Lie algebra is $\fk=\{X\in\fg : X_{x_0}=0\}$. We fix a Cartan decomposition 
\begin{equation}\label{eq:cartan}
\fg=\fk\oplus\fm,
\end{equation}
and recall that 
the space $\fm=\{X\in\fg : (\nabla X)_{x_0} =0\}$ of infinitesimal transvections at $x_0$ 
is naturally identified with $T_{x_0}M$; in particular, $\dim\fm=\dim M=N$.
The codimension of $\K$-orbits on distance spheres $S(r)\subset M$ is equal to $\operatorname{rank}(M) - 1$, so all these $\K$-actions are transitive in our rank one setting. Thus, the eigenvalues of the second fundamental form $A_r$ of $S(r)$ with respect to the unit outward-pointing normal $\vec n_r$ are constant, and can be computed as follows, see e.g.~\cite[\S 6]{bourguignon-karcher}:
\begin{equation*}
\begin{aligned}
&\begin{cases}
2\cot(2r), & \text{with multiplicity } 2d-1\\
\cot(r), & \text{with multiplicity } 2dn
\end{cases},
 & \text{ if } M=\Kr P^{n+1}, \\[5pt]
&\begin{cases}
2\coth(2r), & \text{with multiplicity } 2d-1\\
\coth(r), & \text{with multiplicity } 2dn
\end{cases},
 & \text{ if } M=\Kr H^{n+1},
\end{aligned}
\end{equation*}
where $d=\dim_\C \Kr$, as before. Thus, the mean curvature of $S(r)\subset M$ is:
\begin{equation}\label{eq:HSr}
H(S(r))=\begin{cases}
(N-1)\cot r-(2d-1)\tan r, & \text{if } M=\Kr P^{n+1},\\
(N-1)\coth r+(2d-1)\tanh r, & \text{if } M=\Kr H^{n+1}.
\end{cases}
\end{equation}
Note that $H(S(r))$ is always decreasing, since $N=\dim M=2d(n+1)$; in particular, the map $r\mapsto H(S(r))$ is a diffeomorphism.
Moreover, we have:
\begin{equation}\label{eq:Ar}
\|A_r\|^2=
\begin{cases}
2dn\cot^2(r)+4(2d-1)\cot^2(2r), & \text{if } M=\Kr P^{n+1},\\
2dn\coth^2(r)+4(2d-1)\coth^2(2r), & \text{if } M=\Kr H^{n+1}.
\end{cases}
\end{equation}

The following is essential to determine if $S(r)$ is $\G$-equivariantly nondegenerate:

\begin{lemma}\label{lemma:killingfields}
A Killing field $X\in\fg$ induces a nonzero Jacobi field $\phi_X=\langle X,\vec n_r\rangle$ on $S(r)$ if and only if $X\in\fm$. Thus, the space of Jacobi fields on $S(r)$ has dimension $\geq N$, and equality holds if and only if $S(r)$ is $\G$-equivariantly nondegenerate.
\end{lemma}

\begin{proof}
Clearly, $X\in\fg$ induces the trivial Jacobi field $\phi_X\equiv 0$ if and only if $X$ is everywhere tangent to $S(r)\subset M$. This implies that the $1$-parameter subgroup of isometries in $\G$ associated to such a Killing field $X$ leaves invariant $S(r)=\{x\in M:\dist(x_0,x)=r\}$, and hence fixes $x_0\in M$, so $X\in\fk$. Conversely, $\phi_X\equiv0$ for all $X\in\fk$. The result now follows from \eqref{eq:cartan} and the fact that $\dim \fm=N$.
\end{proof}

Lastly, routine computations of the Einstein constants for these spaces give:
\begin{equation}\label{eq:Ric}
\Ric_{\Kr P^{n+1}}=2dn+4(2d-1), \quad \text{and} \quad \Ric_{\Kr H^{n+1}}=-2dn-4(2d-1).
\end{equation}

We now combine Theorem~\ref{mainthm:A} with Theorems~\ref{thm:IFT} and~\ref{thm:g-bif} to prove Theorem~\ref{mainthm:B}. 

\begin{proof}[Proof of Theorem~\ref{mainthm:B}]
The Jacobi operator $J_r$ of the distance sphere $S(r)\subset M$ can be computed using \eqref{eq:jacobi}, \eqref{eq:Ar}, and \eqref{eq:Ric}, and simplifies to
\begin{equation}\label{eq:Jr}
J_r(\phi)=\Delta_r\phi-V(r)\phi,
\end{equation}
where $\Delta_r=\frac{1}{\alpha^2}\Delta_{\g(t)}$, with $\alpha$ and $t$ as in \eqref{eq:t&alpha}, and 
\begin{equation*}
V(r)=\begin{cases}
(N-1)\csc^2 r+(2d-1)\sec^2 r, & \text{if } M=\Kr P^{n+1},\\
(N-1)\csch^2 r-(2d-1)\sech^2 r, & \text{if } M=\Kr H^{n+1}.
\end{cases}
\end{equation*}

First, let us analyze the projective case $M=\Kr P^{n+1}$, where $\alpha=\sin r$ and $t=\cos r$. By Theorem~\ref{mainthm:A}, the eigenvalues of $\alpha^2 J_r$ are:
\begin{equation*}
\begin{aligned}
\lambda^{(p,q)}(t)-\alpha^2 V(r)&= 4p\big(p+q+\tfrac{N}{2}-1\big)+2dnq+q(q+2d-2)\sec^2 r - V(r)\sin^2 r\\
&= 4p\big(p+q+\tfrac{N}{2}-1\big)+2dnq-(N-1)\\
&\quad +(q(q-1)+(2d-1)(q-\sin^2 r))\sec^2 r,
\end{aligned}
\end{equation*}
for all $(p,q)\in\N_0^2\setminus\{(0,0)\}$.
In particular, for all $p\in \N$,
\begin{equation*}
\begin{aligned}
\lambda^{(p,0)}(\cos r)-\sin^2 r \, V(r) &= 4p\big(p+\tfrac{N}{2}-1\big)-(N-1)-(2d-1)\tan^2 r\\
&= 4p(p-1) +  N(2p-1)+1-(2d-1)\tan^2 r
\end{aligned}
\end{equation*}
is a decreasing function of $0<r<\tfrac{\pi}{2}$, with a unique zero at:
\begin{equation}\label{eq:rp}
r_p:=\arctan\sqrt{\frac{4p(p-1) +  N(2p-1)+1}{2d-1}}.
\end{equation}
Note that $r_1=\arctan\sqrt\frac{N+1}{2d-1}$, 
and $r_p\nearrow \tfrac{\pi}{2}$ as $p\nearrow+\infty$. Moreover, for all $r$,
\begin{equation}\label{eq:l01ker}
\lambda^{(0,1)}(\cos r)-\sin^2 r \, V(r) = 0,
\end{equation}
while, if $q\geq2$, then 
\begin{equation*}
\begin{aligned}
\lambda^{(0,q)}(\cos r)-\sin^2 r \, V(r) &\geq \lambda^{(0,2)}(\cos r)-\sin^2 r \, V(r) \\
&= 2dn +(2d+1) \sec^2 r\\
&\geq N+1>0,
\end{aligned}
\end{equation*}
and, if both $p\geq1$ and $q\geq1$, then
\begin{equation*}
    \begin{aligned}
    \lambda^{(p,q)}(\cos r)-\sin^2 r \, V(r) &\geq 4\big(\tfrac{N}{2}+1\big)+2dn-(N-1)+(2d-1)\\
    &=2N+4>0.
    \end{aligned}
\end{equation*}
Thus, if $r\notin \{r_p:p\in\N\}$, the only zero eigenvalues of $J_r$ are \eqref{eq:l01ker}, and hence $\dim\ker J_r$ coincides with the dimension of the eigenspace of $\Delta_{\g(t)}$ associated to $\lambda^{(0,1)}(t)$, which is $m_{0,1}=N$, by Theorem~\ref{mainthm:A}. Therefore, it follows from Lemma~\ref{lemma:killingfields} that $S(r)$ is $\G$-equivariantly nondegenerate for all $r\notin \{r_p:p\in\N\}$.

Furthermore, it follows from the above spectral analysis that
\begin{equation*}
\imorse(\x_r)=\sum_{ \{p\in\N : r_p < r\} } m_{p,0}.
\end{equation*}
Thus, the claims in Theorem~\ref{mainthm:B} regarding $M=\Kr P^{n+1}$ follow from applying 
Theorem~\ref{thm:IFT} to each $r_*\notin\{r_p:p\in\N\}$, and Theorem~\ref{thm:g-bif} to each $r_*\in\{r_p:p\in\N\}$.

Second, let us analyze the hyperbolic case $M=\Kr H^{n+1}$, where $\alpha=\sinh r$ and $t=\cosh r$. Similarly to the above case, by Theorem~\ref{mainthm:A}, the eigenvalues of $\alpha^2 J_r$ are:
\begin{equation*}
\begin{aligned}
\lambda^{(p,q)}(t)-\alpha^2\, V(r)&= 4p\big(p+q+\tfrac{N}{2}-1\big)+2dnq+q(q+2d-2)\sech^2 r\\
&\quad - V(r)\sinh^2 r\\
&= 4p\big(p+q+\tfrac{N}{2}-1\big)+2dnq-(N-1)\\
&\quad +(q(q-1)+(2d-1)(q+\sinh^2 r))\sech^2 r,
\end{aligned}
\end{equation*}
for all $(p,q)\in\N_0^2\setminus\{(0,0)\}$.
In particular, we have that, for all $r$,
\begin{equation}\label{eq:l01ker-hyp}
\lambda^{(0,1)}(\cosh r)-\sinh^2 r \, V(r) = 0,
\end{equation}
while, if $q\geq 2$, then
\begin{equation*}
\begin{aligned}
\lambda^{(0,q)}(\cosh r)-\sinh^2 r\, V(r) &\geq \lambda^{(0,2)}(\cosh r)-\sinh^2 r\, V(r)\\
&=2dn+(2d+1)\sech^2 r\\
&\geq 2dn>0,
\end{aligned}
\end{equation*}
and, for all $p\geq1$ and $q\in\N_0$,
\begin{equation*}
\lambda^{(p,q)}(\cosh r)-\sinh^2 r\, V(r)
\geq N+1 >0.
\end{equation*}
Thus, the only zero eigenvalues of $J_r$ are \eqref{eq:l01ker-hyp}, and all other eigenvalues are strictly positive, so $\imorse(\x_r)=0$ for all $r>0$, i.e., $S(r)$ is stable for all $r>0$. As before, $\dim\ker J_r$ coincides with the dimension of the eigenspace of $\Delta_{\g(t)}$ associated to $\lambda^{(0,1)}(t)$, which is $m_{0,1}=N$, by Theorem~\ref{mainthm:A}; so Lemma~\ref{lemma:killingfields} implies that $\x_r$ is $\G$-equivariantly nondegenerate for all $r>0$, hence non-resonant by Theorem~\ref{thm:IFT}.
\end{proof}

\begin{remark}
It was known that a sequence of resonant radii $r_p\nearrow\frac{\pi}{2}$ existed for distance spheres $S(r)$ in $\C P^{n+1}$ and $\Hr P^{n+1}$ centered at any point $x_0$ due to basic eigenvalues for the Riemannian submersion $\Ss^{2d-1}\to S(r)\to \mathrm{Cut}(x_0)$, see~\cite[Ex.~6.1]{bp-imrn}. However, neither their exact location \eqref{eq:rp} nor the fact that \emph{only} basic eigenvalues give rise to such bifurcations was previously known. Moreover, the study of local rigidity and resonance for geodesic spheres in $\Ca P^2$ was also not possible in \cite{bp-imrn} since none of the group normality assumptions $\H\triangleleft\K$ or $\K\triangleleft\G$ are satisfied in this case. 
The fact that it was possible to overcome these difficulties in Theorem~\ref{mainthm:B} might suggest that a different approach, e.g., using Mean Curvature Flow, cf.~\cite[Rem.~2.13]{bp-imrn}, may lead to even more general bifurcation results.
\end{remark}

\end{document}